\pdfoutput=1
\documentclass[12pt,a4paper]{article}

\usepackage[T1]{fontenc}
\usepackage[utf8]{inputenc}
\usepackage{lmodern}
\usepackage{amsmath,amssymb,amsthm,mathtools}
\usepackage{geometry}
\usepackage{booktabs}
\usepackage{graphicx}
\usepackage{float}
\usepackage{array}
\usepackage{microtype}
\usepackage{xcolor}
\usepackage[hidelinks]{hyperref}

\numberwithin{equation}{section}

\newtheorem{theorem}{Theorem}[section]
\newtheorem{lemma}[theorem]{Lemma}
\newtheorem{proposition}[theorem]{Proposition}
\newtheorem{corollary}[theorem]{Corollary}

\newcommand{\Z}{\mathbb Z}

\newcommand{\Circ}{\operatorname{Circ}}

\newcommand{\symdiff}{\triangle}

\title{\bfseries $\ell_{1}$-Embeddability of Regular Quadrilateral Maps on the Torus and the Klein Bottle}
\author{Guangfu Wang$^{1,2}\footnote{Corresponding author\\
		Emails: gfwang@ytu.edu.cn,1426654800@qq.com, xxuchunxiao@163.com \\}$,  Wenchao Cong$^1$, Chunxiao Xu$^1$\\[2mm]
	\small 1. School of Mathematics and Information Sciences, Yantai University,\\
	\small Yantai, Shandong 264005, China\\
	\small 2. ECOPRO, Institute for Basic Science, \\
	\small 55 Expo-ro, Yuseong-gu, Daejeon, 34126, Korea}
\date{}

\begin{document}
	\maketitle
	
	\begin{abstract}
		A connected graph is called $\ell_1$-embeddable if its shortest-path metric admits an isometric embedding into the $\ell_1$-space; equivalently, in the finite case, it admits a scale embedding into a hypercube.  This paper gives a complete parameter-level classification of $\ell_1$-embeddability for regular quadrilateral maps on the torus and the Klein bottle.  We treat the two toroidal families $Q_{k,m,r}$ and $Q_{k,m,e}$, obtained from suitable identifications of finite square grids, and, using the known isomorphisms among Klein-bottle quadrilateral maps, the three essential Klein-bottle families $Q_{k,m,a}$, $Q_{k,m,b}$, and $Q_{k,m,h}$.  The $\ell_1$-embeddable toroidal maps are exactly 
		$Q_{k,m,0},  Q_{k,2,1},  Q_{k,1,e},  Q_{5,2,e},  Q_{2i,2,e}  (i\ge 2)$, 
		and the $\ell_1$-embeddable Klein-bottle maps are exactly $
		Q_{2,2,a},  Q_{4,2,a},  Q_{2,m,b},  Q_{1,2,h}.$
		All other members of these five families are not $\ell_1$-embeddable.  
	\end{abstract}
	
	
	\section{Introduction}
	
	Isometric representations of graph metrics in $\ell_1$-spaces lie at the intersection of metric geometry, cut theory, and structural graph theory.  A connected graph $G$ is $\ell_1$-embeddable when its shortest-path metric admits an isometric embedding into the $\ell_1$-space.  For finite graphs this is equivalent to the existence of a scale embedding into a hypercube: there are positive integers $\lambda,n$ and a map $\varphi:V(G)\to V(Q_n)$ such that
	\[
	d_{Q_n}(\varphi(u),\varphi(v))=\lambda d_G(u,v)\qquad (u,v\in V(G)).
	\]
	The scale-one case is precisely the class of partial cubes; see, for example \cite{DL1997}.  Thus an $\ell_1$-embedding represents graph distances as sums of cut metrics, and it imposes strong restrictions on geodesics, isometric cycles, and edge cuts.
	
	Maps on surfaces provide a natural setting in which to study the global nature of this property.  A regular quadrilateral map is locally modeled on the square lattice: every face is a quadrilateral and, away from small degenerate parameters, every vertex has degree four.  Globally, however, the map is obtained by identifying the boundary of a finite grid.  Translational identifications produce toroidal quotients, whereas orientation-reversing identifications produce Klein-bottle quotients.  The resulting graphs can have the same local square-grid geometry but very different systems of metric cuts.  The central question is therefore which global identifications preserve the cut structure required by an $\ell_1$-metric.
	
	A substantial literature has developed around $\ell_1$-embeddability for planar and surface graphs.  Broad criteria are known in the planar case \cite{PSC1990,CDG1997}.  Exact classifications have also been obtained for several families arising from tilings and maps on surfaces, including regular hexagonal and quadrilateral M\"obius graphs \cite{wangzhang2011}, generic quadrilateral M\"obius maps \cite{WS2019}, open-ended nanotube graphs \cite{ZW2010}, shifted quadrilateral cylinder graphs \cite{wangxiongchen2023}, and finite closed polyhexes \cite{DS2009}.  These results consistently show that the gluing pattern, and not merely the local tiling, controls $\ell_1$-embeddability.
	
	The regular maps considered here arise from Thomassen's classification of regular tilings on the torus and the Klein bottle \cite{Thomassen1991}.  In the quadrilateral case there are two toroidal families, denoted here by $Q_{k,m,r}$ and $Q_{k,m,e}$, and six Klein-bottle families.  Known isomorphisms among the latter reduce the metric classification to the three essential families $Q_{k,m,a}$, $Q_{k,m,b}$, and $Q_{k,m,h}$ \cite{XuQian}.  We determine exactly which members of these five essential families are $\ell_1$-embeddable.
	
	Throughout the paper all graphs are finite, connected, and undirected.  When a small parameter value creates loops or parallel edges in the map construction, $\ell_1$-embeddability is understood for the resulting underlying simple graph.  With the parameter ranges specified in Sections~\ref{sec:torus} and~\ref{sec:klein}, the main classification is as follows.
	
	\begin{center}
		\begin{tabular}{lll}
			\toprule
			Surface & Family & Precisely the $\ell_1$-embeddable cases \\
			\midrule
			Torus & $Q_{k,m,r}$ & $r=0$; or $m=2$ and $r=1$ \\
			Torus & $Q_{k,m,e}$ & $m=1$; or $(k,m)=(5,2)$; or $m=2$ and $k$ is even, $k\ge4$ \\
			Klein bottle & $Q_{k,m,a}$ & $(k,m)=(2,2)$ or $(4,2)$ \\
			Klein bottle & $Q_{k,m,b}$ & $k=2$ \\
			Klein bottle & $Q_{k,m,h}$ & $(k,m)=(1,2)$ \\
			\bottomrule
		\end{tabular}
	\end{center}
	
	The remaining Klein-bottle families follow from the isomorphisms cited above.  The classification contains several infinite positive families, but these are highly constrained.  The untwisted toroidal family $Q_{k,m,0}$ is the Cartesian product $C_k\square C_m$, and the only shifted family of the same type that survives is the width-two unit shift $Q_{k,2,1}$.  In the circulant family, the positive cases reduce to cycles, the exceptional complete graph $Q_{5,2,e}\cong K_5$, and the even squared-cycle family $Q_{k,2,e}$$(k \ge 4, even)$.  On the Klein bottle, the only infinite family is the product-like family $Q_{2,m,b}\cong P_2\square C_m$; all other positive cases are small exceptional graphs.
	
	The proof combines two complementary mechanisms.  The positive cases are established by Cartesian-product closure, explicit scale embeddings, and graph isomorphisms to known $\ell_1$-graphs.  The negative cases are driven by Shpectorov's edge-label method.  Given a scale embedding into a hypercube, each edge receives the set of coordinates on which its endpoints differ.  On an isometric even cycle, opposite edges have equal labels; on an isometric odd cycle, opposite labels have prescribed half-intersections; and on a shortest path the successive edge labels are pairwise disjoint.  In the regular quadrilateral families, facial 4-cycles propagate labels across the grid.  For almost every non-trivial shift or twist, this propagation closes inconsistently and forces equal labels on adjacent edges or a graph-distance contradiction.  The remaining circulant cases are handled by the 5-gonal inequality or by a direct label obstruction.
	
	This argument gives a conceptual explanation for the classification. The square grid carries two natural systems of
	cuts, corresponding to its two coordinate directions. Untwisted product identifications preserve those systems. Most shifted
	or orientation-reversing identifications do not: after one travels around a noncontractible direction, the induced edge labels
	return with an incompatible alignment. The exceptional width-two and small-parameter cases survive because the quotient
	collapses to a previously known $l_1$-graph rather than retaining the generic grid geometry.
	
	The toroidal circulant family also yields a concrete graph-theoretic consequence. Since $Q_{k,m,e}$ is isomorphic to Circ$(k,{\pm1,\pm m})$, our theorem classifies this entire two-generator subfamily (for $1\leq m\leq \lfloor k/2 \rfloor)$: the $l_1$-embeddable members are
	exactly the cycle cases $m = 1$, the graph Circ$(5,{\pm1,\pm})\cong K_5$, and Circ$(2i,{\pm1,\pm2})$ for $i\geq 2$. This formulation isolates
	the parity phenomenon in the squared-cycle subfamily and corrects the impression that the conclusion concerns arbitrary
	circulant graphs.
	
	The paper is organized as follows. Section~\ref{sec:prelim} collects the label tools and the 5-gonal inequality.  Section~\ref{sec:torus} treats the two toroidal families and derives the circulant corollary.  Section~\ref{sec:klein} analyzes the three essential Klein-bottle families and completes the classification.  Section~\ref{sec:conclusion} summarizes the structural consequence of the classification  and indicates directions for extending the method to other regular maps and grid quotients.

	\section{Preliminaries}\label{sec:prelim}
	
	For a connected graph $G$, write $d_G$ for its shortest-path metric.  The Cartesian product of graphs is denoted by $\square$.  The $n$-dimensional hypercube $Q_n$ has vertex set $2^{[n]}$, with two vertices adjacent when their symmetric difference has size one.  Thus
	\[
	d_{Q_n}(A,B)=|A\symdiff B|\qquad (A,B\subseteq [n]).
	\]
	
	A finite graph $G$ is an $\ell_1$-graph if and only if it admits a scale embedding into a hypercube; this equivalence is classical and follows from the cut-metric representation of finite $\ell_1$ metrics \cite{AD1980,DL1997}.  Thus there are positive integers $\lambda,n$ and a map $\varphi:V(G)\to V(Q_n)$ such that
	\begin{equation}\label{eq:scale-embedding}
		|\varphi(u)\symdiff \varphi(v)|=\lambda d_G(u,v)\qquad (u,v\in V(G)).
	\end{equation}
	The integer $\lambda$ is called the scale.  For an edge $uv\in E(G)$, following Shpectorov \cite{Shpectorov1993}, define its label, or atom, by
	\[
	\ell(uv)=\varphi(u)\symdiff \varphi(v).
	\]
	Then $|\ell(uv)|=\lambda$ for every edge $uv$.  We shall use the following standard facts about scale embeddings and hypermetric inequalities.  They are quoted here for reference; the new arguments of this paper begin in Section~\ref{sec:torus}.
	
	\begin{lemma}[ \cite{Shpectorov1993,DL1997}]\label{lem:path-labels}
		Let $G$ be an $\ell_1$-graph with a scale-$\lambda$ embedding $\varphi$ into a hypercube.  If $P=u_0u_1\cdots u_s$ is a path in $G$, then
		\[
		\varphi(u_0)\symdiff\varphi(u_s)=\ell(u_0u_1)\symdiff\ell(u_1u_2)\symdiff\cdots\symdiff\ell(u_{s-1}u_s).
		\]
		If $P$ is a shortest path, then the labels $\ell(u_0u_1),\ldots,\ell(u_{s-1}u_s)$ are pairwise disjoint.
	\end{lemma}
	
	\begin{lemma}[ \cite{Shpectorov1993,wangzhang2011}]\label{lem:incident}
		In a simple $\ell_1$-graph, two distinct incident edges have distinct labels.
	\end{lemma}
	
	\begin{lemma}[Isometric cycles \cite{DL1997,Koolen2004}]\label{lem:isometric-cycle}
		Let $C$ be an isometric cycle in an $\ell_1$-graph $G$ with scale $\lambda$.
		\begin{enumerate}
			\item If $C$ is even, then opposite edges of $C$ have equal labels.
			\item If $C$ is odd, then the labels of opposite edges have intersection of size $\lambda/2$; moreover, the label of an edge is contained in the union of the labels of the two opposite edges.
			\item Labels of non-opposite edges of $C$ are disjoint.
		\end{enumerate}
	\end{lemma}
	
	\begin{lemma}[ \cite{DL1997,Koolen2004}]\label{lem:face-labels}
		If $uvwxu$ is an isometric 4-cycle in an $\ell_1$-graph, then
		\[
		\ell(uv)=\ell(wx),\qquad \ell(vw)=\ell(xu),
		\]
		and labels of adjacent edges of this 4-cycle are disjoint.
	\end{lemma}
	
	\begin{lemma}[ \cite{DL1997,Shpectorov1993}]\label{lem:product}
		Cycles and complete graphs are $\ell_1$-graphs.  Moreover, the Cartesian product of two $\ell_1$-graphs is an $\ell_1$-graph.
	\end{lemma}
	
	\begin{lemma}[ \cite{DeGr97}]\label{lem:aprisms}
		For every $b\ge3$, the 4-regular bifaced polyhedron $APrism_b=(4;3,b;2b,2)$ admits a scale-$2$ embedding into a hypercube $Q_{b+1}$A.
	\end{lemma}
	
	\begin{lemma}[ \cite{TY1960,DL1997}]\label{lem:5gonal}
		If $G$ is an $\ell_1$-graph, then for any five vertices $x,y,a,b,c$ of $G$,
		\[
		d(x,y)+d(a,b)+d(a,c)+d(b,c)
		\le d(x,a)+d(x,b)+d(x,c)+d(y,a)+d(y,b)+d(y,c).
		\]
	\end{lemma}
	
	\section{Toroidal quadrilateral maps}\label{sec:torus}
	
	We first treat the two toroidal families in Thomassen's classification.  In $Q_{k,m,r}$, the index $r$ is taken  modulo $k$; in $Q_{k,m,e}$, the index $m$ is taken modulo $k$.
	
	For integers $k,m\ge2$ and $0\le r\le \lfloor k/2\rfloor$, let $Q_{k,m,r}$ have vertex set
	\[
	\{a_{p,j}:0\le p\le m-1,\ j\in\Z_k\}.
	\]
	Its edges are
	\[
	a_{p,j}a_{p,j+1}\quad(0\le p\le m-1),\qquad
	a_{p,j}a_{p+1,j}\quad(0\le p<m-1),
	\]
	together with the twisted boundary edges
	\[
	a_{0,j}a_{m-1,j+r}\qquad (j\in\Z_k).
	\]
	The case $r=0$ is the Cartesian product $C_k\square C_m$.  In fact, the mapping  $f: V(Q_{k,m,r}) \to V(Q_{k,m,k-r}$)  defined by $f(a_{n,i}) = a_{n,(k-i-1)}$ ($n= 0,1,2,\dots,m-1$ and $i=0,1,2,\dots,k-1$)  is an isomorphism between $Q_{k,m,r}$ and $Q_{k,m,k-r}$, so the displayed range for $r$ is sufficient.
	
	For $k\ge3$ and $1\le m\le \lfloor k/2\rfloor$, the graph $Q_{k,m,e}$ has vertex set $\Z_k$ and edges joining $i$ to $i\pm1$ and to $i\pm m$.  Thus
	\begin{equation}\label{eq:circulant}
		Q_{k,m,e}\cong \Circ(k,\{\pm1,\pm m\}).
	\end{equation}
	If $m=1$, parallel edges are suppressed and the underlying simple graph is the cycle $C_k$.
	
	Figure~\ref{fig:torus-constructions} gives the two toroidal constructions used throughout this section.  The left panel shows the finite square grid used as a fundamental domain for $Q_{k,m,r}$; the shifted boundary edges that occur when $r\ne0$ are illustrated in Figure~\ref{fig:qkmr-propagation}.  The right panel shows the circulant model for $Q_{k,m,e}$.
	
	\begin{figure}[htbp]
		\centering
		\begin{minipage}{0.55\textwidth}
			\centering
			\includegraphics[width=\linewidth]{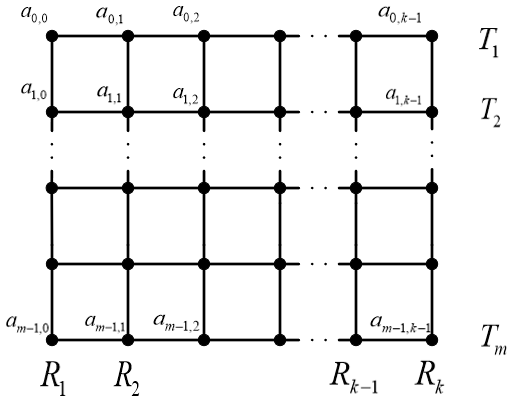}\\[-1mm]
			{\small (a) Grid model for $Q_{k,m,r}$.}
		\end{minipage}\hfill
		\begin{minipage}{0.38\textwidth}
			\centering
			\includegraphics[width=\linewidth]{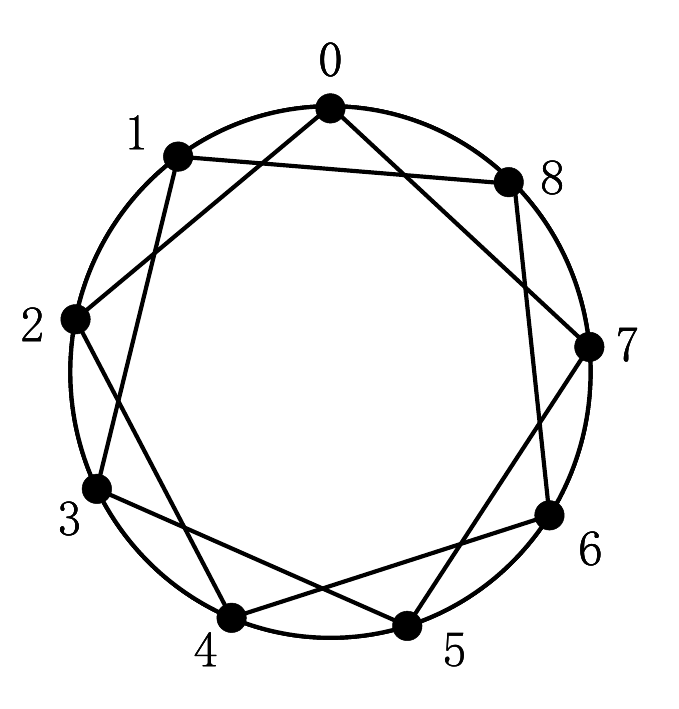}\\[-1mm]
			{\small (b) Example $Q_{9,2,e}=\operatorname{Circ}(9,\{\pm1,\pm2\})$.}
		\end{minipage}
		\caption{Torus models.  In (a), the ordinary square grid is shown before the boundary edges are inserted; adding untwisted cyclic edges in one direction and shifted edges in the other gives $Q_{k,m,r}$.  The graph $Q_{k,m,0}$ is the untwisted product $C_k\square C_m$.  In (b), $Q_{k,m,e}$ is viewed as a two-generator circulant graph.}
		\label{fig:torus-constructions}
	\end{figure}
	
	\subsection[The family Qkmr]{The family $Q_{k,m,r}$}
	
	\begin{proposition}\label{prop:qkmr-positive}
		The graph $Q_{k,m,r}$ is $\ell_1$-embeddable if $r=0$, or if $m=2$ and $r=1$.
	\end{proposition}
	
	\begin{proof}
		If $r=0$, then $Q_{k,m,0}\cong C_k\square C_m$.  Cycles are $\ell_1$-graphs and Cartesian products preserve $\ell_1$-embeddability by Lemma~\ref{lem:product}.
		
		Now assume $m=2$ and $r=1$.  If $k=2$, then the underlying simple graph is $K_4$, hence is $\ell_1$-embeddable by Lemma~\ref{lem:product}.  If $k\ge3$, then $Q_{k,2,1}$ is the 4-regular bifaced polyhedron $APrism_k=(4;3,k;2k,2)$ in the notation of Deza and Grishukhin, and Lemma~\ref{lem:aprisms} applies.  Therefore $Q_{k,2,1}$ is an $\ell_1$-graph.
	\end{proof}
	
	We shall use the following elementary distance facts for the quotient grid.  The graph $Q_{k,m,r}$ is the quotient of the square lattice by the translations $(0,k)$ and $(m,r)$.  Therefore the distance between two vertices represented by lifts $(p,j)$ and $(p',j')$ is
	\begin{equation}\label{eq:quotient-distance}
		\min_{s,t\in\Z}\bigl(|p'-p+sm|+|j'-j+sr+tk|\bigr).
	\end{equation}
	A direct evaluation of \eqref{eq:quotient-distance} gives the following facts used below: all facial 4-cycles occurring in the grid are isometric; when $ r=\lfloor k/2\rfloor$, the cycle
	\begin{equation}\label{eq:long-cycle}
		C=a_{0,0}a_{0,1}\cdots a_{0,r}a_{1,r}\cdots a_{m-1,r}a_{0,0}
	\end{equation}
	is isometric; and in the special case $k \ge 2, m \ge 3,r=1$ or $k \ge 2, m \ge 2, 2 \le r \le \lfloor k/2 \rfloor$,   the boundary 4-cycles
	\[
	a_{0,i}a_{0,i+1}a_{m-1,i+2}a_{m-1,i+1}a_{0,i} ~( 0\le i \le k-1)
	\]
	are isometric.
	
	The two label-propagation patterns used in the proof are illustrated in Figure~\ref{fig:qkmr-propagation}.  The highlighted curves should be read schematically: each highlighted edge represents the same Shpectorov label after repeated applications of Lemma~\ref{lem:face-labels} to isometric quadrilateral faces.
	
	\begin{figure}[htbp]
		\centering
		\begin{minipage}{0.48\textwidth}
			\centering
			\includegraphics[width=\linewidth]{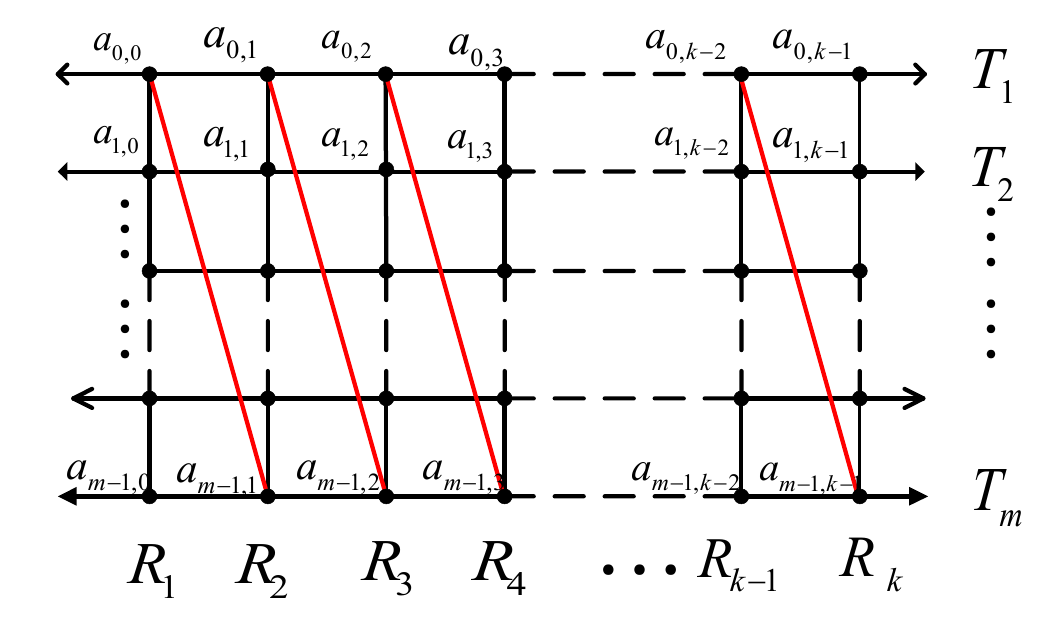}\\[-1mm]
			{\small (a) Unit shift $r=1$.}
		\end{minipage}\hfill
		\begin{minipage}{0.48\textwidth}
			\centering
			\includegraphics[width=\linewidth]{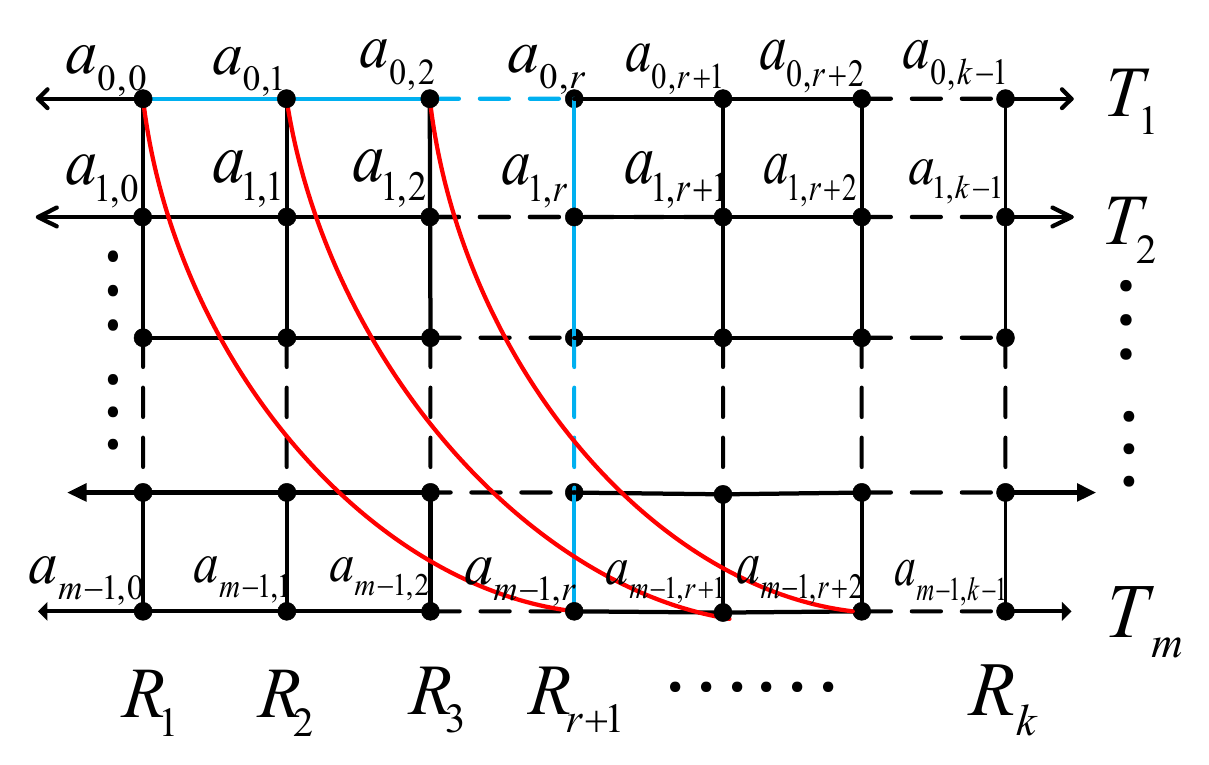}\\[-1mm]
			{\small (b) The long cycle for $2\le r\le\lfloor k/2\rfloor$.}
		\end{minipage}
		\caption{Label propagation in $Q_{k,m,r}$.  Panel (a) shows why, when $m\ge3$, the shift $r=1$ propagates a horizontal label until two incident edges receive the same label.  Panel (b) shows the shift $2\le r\le\lfloor k/2\rfloor$ and the isometric cycle $C$ used for $r= \lfloor k/2 \rfloor$; an opposite-edge relation on $C$ contradicts the disjointness forced by the intervening facial 4-cycles.}
		\label{fig:qkmr-propagation}
	\end{figure}
	
	\begin{theorem}\label{thm:qkmr}
		For $k,m\ge2$ and $0\le r\le \lfloor k/2\rfloor$, the graph $Q_{k,m,r}$ is an $\ell_1$-graph if and only if $r=0$, or $m=2$ and $r=1$.
	\end{theorem}
	
	\begin{proof}
		The positive direction is Proposition~\ref{prop:qkmr-positive}.  Conversely, assume that $Q_{k,m,r}$ is $\ell_1$-embeddable with scale $\lambda$, and suppose that it is not one of the positive cases.
		
		First let $r=1$.  Then $m\ge3$.  If $k=2$, choose
		\[
		x=a_{0,0},\quad y=a_{1,1},\quad a=a_{0,1},\quad b=a_{1,0},\quad c=a_{s,1},
		\qquad s=\Big\lfloor\frac{m+1}{2}\Big\rfloor.
		\]
		A direct distance calculation gives a violation of Lemma~\ref{lem:5gonal}: if $m=2q$, then the two sides are $2q+4$ and $2q+3$, while if $m=2q+1$, they are $2q+6$ and $2q+4$.  Hence this degenerate case is not $\ell_1$-embeddable.
		
		Assume now that $k\ge3$.  By the isometric boundary 4-cycles and facial 4-cycles, Lemma~\ref{lem:face-labels} propagates the horizontal label as follows:
		\[
		\ell(a_{0,0}a_{0,1})=\ell(a_{m-1,1}a_{m-1,2})=\ell(a_{m-2,1}a_{m-2,2})=\cdots=\ell(a_{0,1}a_{0,2}).
		\]
		The two edges $a_{0,0}a_{0,1}$ and $a_{0,1}a_{0,2}$ are adjacent and distinct, but they have the same edge label, which  contradicting Lemma~\ref{lem:incident}.  Thus no graph with $r=1$ and $m\ge3$ is $\ell_1$-embeddable.
		
		It remains to consider $ m \ge 2, 2\le r\le\lfloor k/2\rfloor-1$, which induced $k \ge 6$.  
		
		The facial 4-cycles imply
		\begin{equation}\label{eq:qkmr-horizontal-propagation}
			\ell(a_{0,0}a_{0,1})=\ell(a_{1,0}a_{1,1})=\cdots=\ell(a_{m-1,0}a_{m-1,1}),
		\end{equation}
		and, for $1\le j\le m-1$,
		\begin{equation}\label{eq:qkmr-vertical-propagation}
			\ell(a_{j-1,0}a_{j,0})=\ell(a_{j-1,1}a_{j,1})=\cdots=\ell(a_{j-1,r}a_{j,r}).
		\end{equation}
		
		Let
		\begin{equation}
			\ell(a_{0,0}a_{0,1})=A, \ell(a_{0,1}a_{0,2})=B,\dots, \ell(a_{0,r-1}a_{0,r})=R.
		\end{equation}
		By the isometric boundary 4-cycles and facial 4-cycles, Lemma~\ref{lem:face-labels}
		
		\begin{equation*}
			\begin{aligned}
				\ell(a_{0,0}a_{0,1})=\ell(a_{m-1,r}a_{m-1,r+1})=&\ell(a_{m-2,r}a_{m-2,r+1})=\dots= \ell(a_{0,r}a_{0,r+1})=A,\\
				\ell(a_{0,1}a_{0,2})=\ell(a_{m-1,r+1}a_{m-1,r+2})=&\ell(a_{m-2,r+1}a_{m-2,r+2})=\dots= \ell(a_{0,r+1}a_{0,r+2})=B,\\
				&\vdots \\
				\ell(a_{0,r-1}a_{0,r})=\ell(a_{m-1,2r-1}a_{m-1,2r})=&\ell(a_{m-2,2r-1}a_{m-2,2r})=\dots= \ell(a_{0,2r-1}a_{0,2r})=R.
			\end{aligned}
		\end{equation*}
		Therefore, by Lemma \ref{lem:path-labels}, we know that:

		\begin{equation*}
			\begin{aligned}
				\left|\phi ({a_{0,0}})\triangle \phi ({a_{0,2r}}) \right|&=\left|\ell({a_{0,0}}{a_{0,1}})\triangle \ell({a_{0,1}}{a_{0,2}}) \triangle  \cdots \triangle \ell({a_{0,2r - 1}}{a_{0,2r}}) \right|\\
				&=\left| A \triangle B   \cdots \triangle R \triangle A \cdots \triangle R\right| \\
				&=0
			\end{aligned}
		\end{equation*}
		which yields a contradiction. By a similar argument, the case  $r=\lfloor k/2\rfloor$ with $k$ is odd can be treated analogously.
		
		Finally consider $ m \ge 2, 1 \ne r=\lfloor k/2\rfloor$, which induced $k \ge 4$.
		
		Since the two-edge paths $a_{j,0}a_{j,1}a_{j+1,1}$ are shortest, Lemma~\ref{lem:path-labels} and \eqref{eq:qkmr-horizontal-propagation}--\eqref{eq:qkmr-vertical-propagation} give
		\begin{equation}\label{eq:qkmr-disjoint-vertical}
			\ell(a_{0,0}a_{0,1})\cap \ell(a_{j,r}a_{j+1,r})=\varnothing\qquad (0\le j\le m-2).
		\end{equation}
		Similarly, using the shortest two-edge paths along the top row and the first vertical column,
		\begin{equation}\label{eq:qkmr-disjoint-horizontal}
			\ell(a_{0,r}a_{1,r})\cap \ell(a_{0,i}a_{0,i+1})=\varnothing\qquad (0\le i\le r-1).
		\end{equation}
		
		Now use the isometric cycle $C$ in \eqref{eq:long-cycle}.  If an edge opposite $a_{0,0}a_{0,1}$ in $C$ lies on the vertical side of $C$, then Lemma~\ref{lem:isometric-cycle} gives a non-empty intersection, of size $\lambda$ when $C$ is even and $\lambda/2$ when $C$ is odd, between $\ell(a_{0,0}a_{0,1})$ and the label of an edge $a_{j,r}a_{j+1,r}$.  This contradicts \eqref{eq:qkmr-disjoint-vertical}.  Otherwise, an edge opposite $a_{0,r}a_{1,r}$ lies on the top side of $C$, and Lemma~\ref{lem:isometric-cycle} contradicts \eqref{eq:qkmr-disjoint-horizontal}.  In both cases we obtain a contradiction.  Hence no remaining $Q_{k,m,r}$ is $\ell_1$-embeddable.
	\end{proof}
	
	\subsection[The family Qkme]{The family $Q_{k,m,e}$}
	
	\begin{proposition}\label{prop:qkme-positive}
		The following graphs in the family $Q_{k,m,e}$ are $\ell_1$-embeddable:
		\[
		Q_{k,1,e},\qquad Q_{5,2,e},\qquad Q_{2i,2,e}\quad (i\ge2).
		\]
	\end{proposition}
	
	\begin{proof}
		The graph $Q_{k,1,e}$ is the cycle $C_k$ after parallel edges are suppressed, and $Q_{5,2,e}$ is $K_5$, because the generators $\pm1$ and $\pm2$ give all non-zero elements of $\Z_5$.  Both are $\ell_1$-graphs by Lemma~\ref{lem:product}.
		
		For $i\ge2$, define a map from $Q_{2i,2,e}$ to $Q_{i,2,1}$ by
		\[
		2t\longmapsto a_{1,t},\qquad 2t+1\longmapsto a_{0,t}\qquad (t\in\Z_i).
		\]
		The edges of $Q_{2i,2,e}$ have differences $\pm1$ and $\pm2$.  Under this map, the $\pm2$ edges become horizontal edges in one of the two rows of $Q_{i,2,1}$, while the $\pm1$ edges become the vertical or shifted boundary edges of $Q_{i,2,1}$.  Thus $Q_{2i,2,e}\cong Q_{i,2,1}$, which is $\ell_1$-embeddable by Proposition~\ref{prop:qkmr-positive}.
	\end{proof}
	
	\begin{proposition}\label{prop:q72e}
		The graph $Q_{7,2,e}$ is not $\ell_1$-embeddable.
	\end{proposition}
	
	\begin{proof}
		In $Q_{7,2,e}$ take $x=0$, $y=3$, $a=1$, $b=2$, and $c=5$.  The left-hand side of the 5-gonal inequality is
		\[
		d(0,3)+d(1,2)+d(1,5)+d(2,5)=2+1+2+2=7,
		\]
		whereas the right-hand side is
		\[
		d(0,1)+d(0,2)+d(0,5)+d(3,1)+d(3,2)+d(3,5)=1+1+1+1+1+1=6.
		\]
		This violates Lemma~\ref{lem:5gonal}.
	\end{proof}
	
	\begin{proposition}\label{prop:q92e}
		The graph $Q_{9,2,e}$ is not $\ell_1$-embeddable.
	\end{proposition}
	
	\begin{proof}
		The graph is displayed in Figure~\ref{fig:torus-constructions}(b).
		Suppose that $Q_{9,2,e}$ admits a scale-$\lambda$ embedding $\varphi$ into a hypercube.  Translate the embedding so that $\varphi(0)=\varnothing$, and put
		\[
		A=\varphi(1),\qquad B=\varphi(2),\qquad C=\varphi(7),\qquad D=\varphi(8).
		\]
		Then $|A|=|B|=|C|=|D|=\lambda$.  The triangles $012$, $018$, and $078$ give
		\begin{equation}\label{eq:q9-triangles}
			|A\cap B|=|A\cap D|=|C\cap D|=\lambda/2.
		\end{equation}
		The paths $1-0-7$, $2-0-7$, and $2-0-8$ are shortest, so Lemma~\ref{lem:path-labels} gives
		\begin{equation}\label{eq:q9-disjoint}
			A\cap C=B\cap C=B\cap D=\varnothing.
		\end{equation}
		From \eqref{eq:q9-triangles}--\eqref{eq:q9-disjoint}, we have
		\begin{equation}\label{eq:q9-D}
			D=(A\setminus B)\cup (C\cap D).
		\end{equation}
		Indeed, $A\cap D\subseteq A\setminus B$ and both sets have size $\lambda/2$.
		
		The paths $0-1-3$ and $0-2-3$ are shortest.  Hence $A\cup B\subseteq\varphi(3)$ and $|\varphi(3)|=2\lambda$.  The cycle $0,2,3,5,7,0$ is an isometric 5-cycle, and the edges $23$ and $70$ are opposite.  Since $B\subseteq\varphi(3)$,
		\[
		\ell(23)=\varphi(2)\symdiff\varphi(3)=\varphi(3)\setminus B.
		\]
		Lemma~\ref{lem:isometric-cycle} therefore gives
		\[
		|\varphi(3)\cap C|=|\ell(23)\cap\ell(70)|=\lambda/2.
		\]
		Let $F=\varphi(3)\cap C$.  Since $C$ is disjoint from $A\cup B$ and $|A\cup B|=3\lambda/2$, we obtain
		\begin{equation}\label{eq:q9-phi3}
			\varphi(3)=(A\cup B)\cup F.
		\end{equation}
		
		Similarly, the paths $0-8-6$ and $0-7-6$ are shortest.  Hence $C\cup D\subseteq\varphi(6)$ and $|\varphi(6)|=2\lambda$.  The cycle $0,2,4,6,8,0$ is an isometric 5-cycle, and the edges $02$ and $68$ are opposite.  Since $D\subseteq\varphi(6)$,
		\[
		\ell(68)=\varphi(6)\setminus D,
		\]
		and Lemma~\ref{lem:isometric-cycle} gives
		\[
		|B\cap\varphi(6)|=|\ell(02)\cap\ell(68)|=\lambda/2.
		\]
		Let $G=B\cap\varphi(6)$.  Since $B$ is disjoint from $C\cup D$ and $|C\cup D|=3\lambda/2$, we obtain
		\begin{equation}\label{eq:q9-phi6}
			\varphi(6)=(C\cup D)\cup G.
		\end{equation}
		Using \eqref{eq:q9-D}, \eqref{eq:q9-phi3}, and \eqref{eq:q9-phi6},
		\[
		\varphi(3)=(A\setminus B)\cup B\cup F,
		\qquad
		\varphi(6)=(A\setminus B)\cup C\cup G.
		\]
		Here $F\subseteq C$ and $G\subseteq B$, with $|F|=|G|=\lambda/2$.  Therefore
		\[
		\varphi(3)\symdiff\varphi(6)=(B\setminus G)\cup(C\setminus F),
		\]
		and so $|\varphi(3)\symdiff\varphi(6)|=\lambda$.  By \eqref{eq:scale-embedding}, this gives $d(3,6)=1$.  But vertices $3$ and $6$ are not adjacent in $Q_{9,2,e}$, since their difference is $3\notin\{\pm1,\pm2\}\pmod 9$.  This contradiction proves the proposition.
	\end{proof}
	
	\begin{proposition}\label{prop:odd-cycle-square}
		For every odd integer $k\ge11$, the graph $Q_{k,2,e}$ is not an $\ell_1$-graph.
	\end{proposition}
	
	\begin{proof}
		Write $k=2i+1$, so $i\ge5$, and denote $Q_{k,2,e}$ by $G_i$.  For $a,b\in\Z_{2i+1}$ let $\rho(a,b)$ be the cyclic distance in $C_{2i+1}$.  Since each edge of $G_i$ changes the cyclic position by $1$ or $2$,
		\begin{equation}\label{eq:qk2-distance}
			d_{G_i}(a,b)=\left\lceil\frac{\rho(a,b)}2\right\rceil.
		\end{equation}
		The lower bound follows because one step reduces cyclic distance by at most two, and the upper bound is obtained by using jumps of length two and, if necessary, one jump of length one along a shortest cyclic arc.
		
		If $i=2q+1$ is odd, choose
		\[
		x=0,\qquad y=3,\qquad a=1,\qquad b=2,\qquad c=i+2=2q+3.
		\]
		Using \eqref{eq:qk2-distance}, the left-hand side of the 5-gonal inequality is
		\[
		L=d(0,3)+d(1,2)+d(1,2q+3)+d(2,2q+3)=2+1+(q+1)+(q+1)=2q+5,
		\]
		whereas the right-hand side is
		\[
		R=d(0,1)+d(0,2)+d(0,2q+3)+d(3,1)+d(3,2)+d(3,2q+3)=1+1+q+1+1+q=2q+4.
		\]
		Thus $L>R$.
		
		If $i=2q$ is even, choose
		\[
		x=0,\qquad y=5,\qquad a=1,\qquad b=4,\qquad c=i+3=2q+3.
		\]
		Again by \eqref{eq:qk2-distance},
		\[
		L=d(0,5)+d(1,4)+d(1,2q+3)+d(4,2q+3)=3+2+q+q=2q+5,
		\]
		and
		
		\[
		\begin{aligned}
			R&=d(0,1)+d(0,4)+d(0,2q+3)+d(5,1)+d(5,4)+d(5,2q+3)\\
			&=1+2+(q-1)+2+1+(q-1)=2q+4.
		\end{aligned}
		\]
		Thus the 5-gonal inequality is violated in both parity cases.
	\end{proof}
	
	\begin{theorem}\label{thm:qkme}
		For $k\ge3$ and $1\le m\le\lfloor k/2\rfloor$, the graph $Q_{k,m,e}$ is $\ell_1$-embeddable if and only if
		\[
		m=1,
		\qquad (k,m)=(5,2),
		\qquad \text{or}\qquad m=2\text{ and }k\text{ is even with }k\ge4.
		\]
	\end{theorem}
	
	\begin{proof}
		The positive cases are Proposition~\ref{prop:qkme-positive}.  Conversely, assume that $Q_{k,m,e}$ is not one of these cases.
		
		If $m=2$, then $k$ is odd and $k\ne5$.  The cases $k=7$ and $k=9$ are Propositions~\ref{prop:q72e} and~\ref{prop:q92e}; the cases $k\ge11$ follow from Proposition~\ref{prop:odd-cycle-square}.
		
		Assume next that $3\le m<\lfloor k/2\rfloor$.  Then $2m<k$.  For each $i\in\Z_k$, the cycle
		\[
		i,\ i+1,
		i+m+1,
		i+m,
		i
		\]
		is an induced, hence isometric, 4-cycle.  Therefore Lemma~\ref{lem:face-labels} gives
		\[
		\ell(0,m)=\ell(1,m+1)=\cdots=\ell(m,2m).
		\]
		The first and last edges in this chain are incident at the vertex $m$, contradicting Lemma~\ref{lem:incident}.
		
		It remains to consider $m=\lfloor k/2\rfloor$.  If $k=2i$ is even, then $m=i$.  The graph $Q_{2i,i,e}$ is the cycle $C_{2i}$ together with the perfect matching joining antipodal vertices.  The case $i=2$ is $Q_{4,2,e}$, already positive, so assume $i\ge3$.  Let $G_i=Q_{2i,i,e}$.  For $0\le t\le i$ one has
		\begin{equation}\label{eq:antipodal-cycle-distance}
			d_{G_i}(0,t)=\min\{t,i+1-t\}.
		\end{equation}
		Indeed, either one follows the outer cycle for $t$ steps, or one uses the antipodal edge $0i$ and then moves $i-t$ steps along the cycle; no shorter walk is possible after reducing multiple antipodal-edge uses.
		
		Put
		\[
		r_0=\left\lfloor\frac{i+1}{2}\right\rfloor,
		\qquad
		a=0,
		\quad b=r_0,
		\quad c=i+1,
		\quad x=1,
		\quad y=i.
		\]
		Using \eqref{eq:antipodal-cycle-distance},
		\[
		d(a,b)=r_0,
		\quad d(a,c)=2,
		\quad d(b,c)=r_0,
		\quad d(x,y)=2,
		\]
		so the left-hand side of the 5-gonal inequality is $L=2r_0+4$.  The six mixed distances are
		\[
		\begin{gathered}
			d(x,a)=1,\quad d(x,b)=r_0-1,\quad d(x,c)=1,\\
			d(y,a)=1,\quad d(y,b)=\left\lceil\frac{i-1}{2}\right\rceil,\quad d(y,c)=1.
		\end{gathered}
		\]
		Thus
		\[
		R=3+r_0+\left\lceil\frac{i-1}{2}\right\rceil.
		\]
		If $i=2q$, then $L=2q+4$ and $R=2q+3$; if $i=2q+1$, then $L=2q+6$ and $R=2q+4$.  Hence Lemma~\ref{lem:5gonal} is violated.
		
		Finally suppose $k=2i+1$ is odd and $m=i\ge3$.  Define a mapping :
		\begin{center}
			$\tau :V(Q_{2i+1,m,e}) \to V(Q_{2i+1,2,e})$ for $i \ge 3$,
		\end{center}
		such that for $t \in \{0,1,2, \dots,k-1\}$
		\begin{center}
			$\tau(t)= 2t-1$ module $k$.
		\end{center}
		Suppose that $u\in V(Q_{2i+1,m,e})$ whose  neighbors are $u+1$, $u-1$, $u+i$ and $u-i$. According to the mapping  $\tau$, we have
		\begin{center}
			$\tau(u)= 2u-1$, $\tau(u-1)= 2u-3$, $\tau(u+1)= 2u+1$, $\tau(u+i)= 2u+2i-1$, $\tau(u-i)= 2u-2i-1$.
		\end{center}
		Since $m =i$ and $k=2i+1$. Therefore,
		\begin{center}
			$\tau(u+i)= 2u-2$ and $\tau(u-i)= 2u$.
		\end{center}
		However, in the graph $Q_{2i+1,2,e}$, the vertex $2u-1$ has neighbors $2u$, $2u-2$, $2u+1$, and $2u-3$. That means that  the adjacency relations in $Q_{2i+1,m,e}$ are preserved in $Q_{2i+1,2,e}$  under the mapping $\tau$. Hence
		\[
		Q_{2i+1,i,e}\cong Q_{2i+1,2,e}.
		\]
		The latter is non-$\ell_1$ by the already proved $m=2$ case.  This completes the proof.
	\end{proof}
	
	Combining Theorems~\ref{thm:qkmr} and~\ref{thm:qkme} gives the toroidal part of the classification.
	
	\begin{corollary}[Toroidal maps]\label{cor:torus}
		Among the toroidal regular quadrilateral maps in the two families above, the $\ell_1$-embeddable graphs are precisely
		\[
		Q_{k,m,0},\qquad Q_{k,2,1},\qquad Q_{k,1,e},\qquad Q_{5,2,e},\qquad Q_{2i,2,e}\ (i\ge2).
		\]
	\end{corollary}
	
	\begin{corollary}\label{cor:circulant}
		Within the family
		\[
		\Circ(k,\{\pm1,\pm m\}),\qquad k\ge3,\qquad
		1\le m\le\lfloor k/2\rfloor,
		\]
		the $\ell_1$-embeddable members are exactly the cycle cases $m=1$, the exceptional graph $\Circ(5,\{\pm1,\pm2\})\cong K_5$, and the graphs $\Circ(2i,\{\pm1,\pm2\})$ for $i\ge2$.
	\end{corollary}
	
	\begin{proof}
		This is exactly Theorem~\ref{thm:qkme} under the isomorphism \eqref{eq:circulant}.
	\end{proof}
	
	\section{Klein-bottle quadrilateral maps}\label{sec:klein}
	
	Thomassen's classification contains six families of 4-regular quadrilateral maps $Q_{k,m,a}$, $Q_{k,m,b}$, $Q_{k,m,c}$, $Q_{k,m,f}$, $Q_{k,m,g}$, and $Q_{k,m,h}$ on the Klein bottle \cite{Thomassen1991}.  We use the following known isomorphisms to reduce the classification to three essential families.
	
	\begin{lemma}[Klein-bottle reductions \cite{XuQian}]\label{lem:klein-reductions}
		The Klein-bottle quadrilateral-map families satisfy
		\[
		Q_{k,m,c}\cong Q_{2m,k/2,a},\qquad
		Q_{k,m,f}\cong Q_{2(m+1),k/2,b},\qquad
		Q_{k,m,g}\cong Q_{2m+1,k/2,a}.
		\]
		Moreover, the odd-$k$ version of the $b$-type boundary gluing is isomorphic to a member of the $a$-family.  Hence it suffices to determine $\ell_1$-embeddability for $Q_{k,m,a}$, $Q_{k,m,b}$ with even $k$, and $Q_{k,m,h}$.
	\end{lemma}
	
	By Lemma~\ref{lem:klein-reductions}, we treat the three essential families $Q_{k,m,a}$, $Q_{k,m,b}$, and $Q_{k,m,h}$ in turn.
	
	Let $C_k\square P_m$ be a quadrilateral cylinder with levels $0,1,\ldots,m-1$ and cyclic coordinate $j\in\Z_k$.  Write $x_j$ for the boundary vertex on level $0$ and $y_j$ for the boundary vertex on level $m-1$; interior vertices are denoted by $u_{p,j}$, $1\le p\le m-2$.
	
	The family $Q_{k,m,a}$, $k,m\ge2$, is obtained from this cylinder by adding the boundary edges $x_1y_2, x_2y_1, x_ky_3, x_3y_k,\ldots$  The family $Q_{k,m,b}$ is needed only for even $k$; it is obtained by adding  
	$x_1y_1, x_2y_k, x_ky_2, \dots, x_{\frac{k}{2}+1}y_{\frac{k}{2}+1}$.  The odd-$k$ case is covered by Lemma~\ref{lem:klein-reductions}.
	
	Figure~\ref{fig:klein-ab} displays representative Klein-bottle gluings for the $a$- and $b$-families.  These drawings are not used as topological embeddings in the proofs; they are combinatorial schematics that make the boundary identifications visible.
	
	\begin{figure}[H]
		\centering
		\begin{minipage}{0.32\textwidth}
			\centering
			\includegraphics[width=\linewidth]{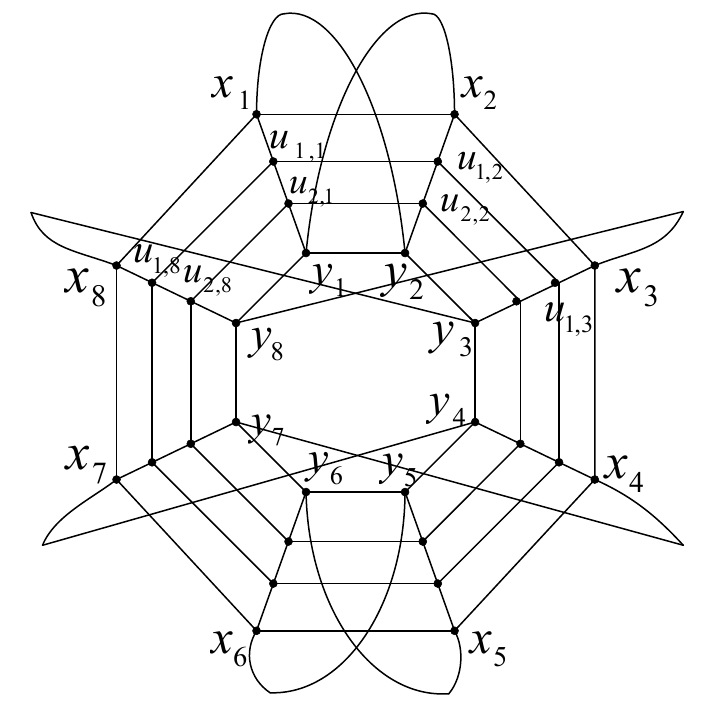}\\[-1mm]
			{\small (a) $Q_{8,4,a}$.}
		\end{minipage}\hfill
		\begin{minipage}{0.32\textwidth}
			\centering
			\includegraphics[width=\linewidth]{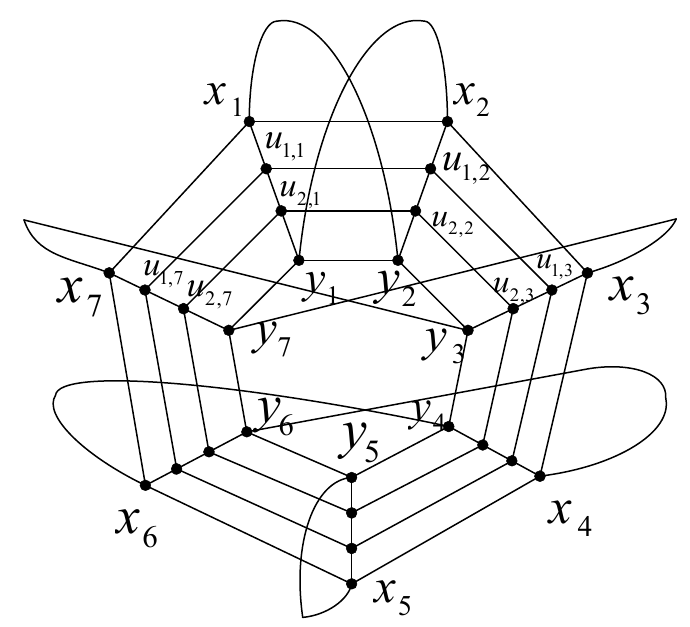}\\[-1mm]
			{\small (b) $Q_{7,4,a}$.}
		\end{minipage}\hfill
		\begin{minipage}{0.32\textwidth}
			\centering
			\includegraphics[width=\linewidth]{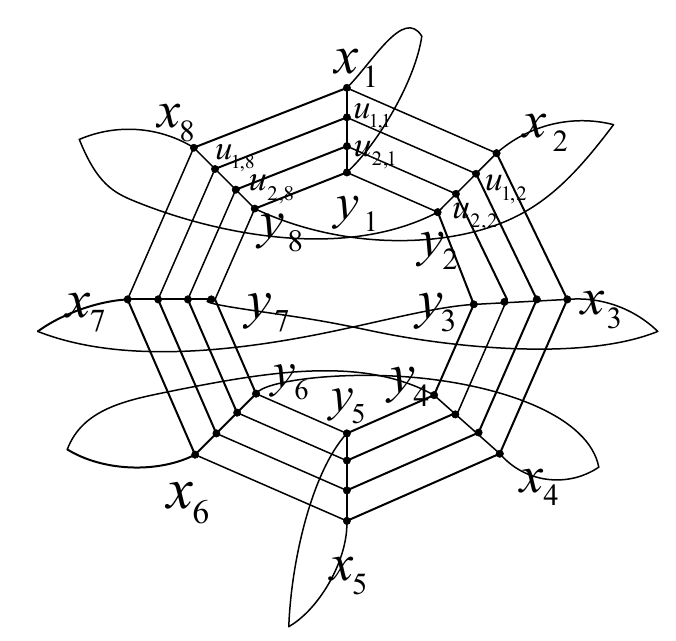}\\[-1mm]
			{\small (c) $Q_{8,4,b}$.}
		\end{minipage}
		\caption{Representative Klein-bottle quadrilateral maps obtained from a quadrilateral cylinder.  The added boundary edges reverse the orientation of one boundary cycle; the $a$- and $b$-families differ in the shift in this orientation-reversing gluing.}
		\label{fig:klein-ab}
	\end{figure}
	
	\subsection[The family Qkma]{The family $Q_{k,m,a}$}
	
	The exceptional graph $Q_{4,2,a}$ and its scale-$2$ hypercube labels are shown in Figure~\ref{fig:q42a-embedding}.  This is the only non-complete positive member of the $a$-family.
	
	\begin{figure}[H]
		\centering
		\includegraphics[width=0.46\textwidth]{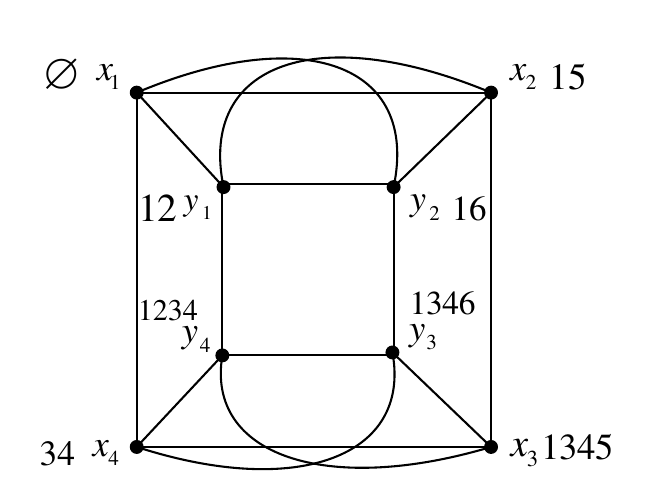}
		\caption{The exceptional map $Q_{4,2,a}$.  The labels at the vertices are subsets of $[6]$ and give a scale-$2$ embedding into $Q_6$; for example, $1345$ denotes $\{1,3,4,5\}$.}
		\label{fig:q42a-embedding}
	\end{figure}
	
	\begin{theorem}\label{thm:qkma}
		For $k,m\ge2$, the graph $Q_{k,m,a}$ is $\ell_1$-embeddable if and only if $(k,m)=(2,2)$ or $(k,m)=(4,2)$.
	\end{theorem}
	
	\begin{proof}
		The graph $Q_{2,2,a}$ is $K_4$ after loops and parallel edges are suppressed, and hence is an $\ell_1$-graph by Lemma~\ref{lem:product}.  For $Q_{4,2,a}$, Figure~\ref{fig:q42a-embedding} and the following table give a scale-2 embedding into $Q_6$:
		\[
		\begin{array}{c|cccccccc}
			v & x_1&x_2&x_3&x_4&y_1&y_2&y_3&y_4\\
			\hline
			\varphi(v) & \varnothing&15&1345&34&12&16&1346&1234
		\end{array}
		\]
		Here, for instance, $1345$ denotes the subset $\{1,3,4,5\}\subseteq[6]$.  A direct check gives
		\[
		|\varphi(u)\symdiff\varphi(v)|=2d_{Q_{4,2,a}}(u,v)
		\]
		for all pairs of vertices $u,v$.  Thus $Q_{4,2,a}$ is $\ell_1$-embeddable.
		
		Conversely, assume that $Q_{k,m,a}$ is $\ell_1$-embeddable with scale $\lambda$, and suppose that it is not one of the two positive cases.
		
		First let $m=2$.  The cycles $x_1x_ky_ky_1x_1$ and $x_2x_3y_ky_1x_2$ are isometric 4-cycles in the underlying simple graph.  Lemma~\ref{lem:face-labels} gives
		\[
		\ell(x_1x_k)=\ell(y_1y_k)=\ell(x_2x_3).
		\]
		If $k=3$, these are two incident top-cycle edges, contradicting Lemma~\ref{lem:incident}.  If $k\ge5$, then along the path $x_kx_1x_2x_3$ the first and last labels cancel.  Hence
		\[
		|\varphi(x_k)\symdiff\varphi(x_3)|
		=|\ell(x_kx_1)\symdiff\ell(x_1x_2)\symdiff\ell(x_2x_3)|
		=|\ell(x_1x_2)|=\lambda.
		\]
		Therefore $d(x_k,x_3)=1$, although $x_k$ and $x_3$ are not adjacent when $k\ge5$.  This is impossible.
		
		Now let $m\ge3$.  The facial 4-cycles of the cylinder are isometric, so horizontal labels propagate through the cylinder:
		\[
		\ell(x_1x_2)=\ell(u_{1,1}u_{1,2})=\cdots=\ell(y_1y_2)=:\eta.
		\]
		For every edge $e$ on the vertical path
		\[
		P=x_1u_{1,1}u_{2,1}\cdots u_{m-2,1}y_1,
		\]
		the label of $e$ is disjoint from $\eta$, by Lemma~\ref{lem:face-labels} applied to the corresponding isometric face.  On the other hand, the cycle
		\[
		x_1u_{1,1}u_{2,1}\cdots u_{m-2,1}y_1y_2x_1
		\]
		is isometric.  In this cycle, the edge $y_1y_2$ has an opposite edge lying on the path $P$.  Lemma~\ref{lem:isometric-cycle} gives a non-empty intersection, of size $\lambda$ in the even case and $\lambda/2$ in the odd case, between $\ell(y_1y_2)=\eta$ and the label of that opposite edge.  This contradicts the disjointness just obtained.  Hence $Q_{k,m,a}$ is not $\ell_1$-embeddable for $m\ge3$.
	\end{proof}
	
	\subsection[The family Qkmb]{The family $Q_{k,m,b}$}
	
	\begin{theorem}\label{thm:qkmb}
		Let $k\ge2$ be even and let $m\ge2$.  The graph $Q_{k,m,b}$ is $\ell_1$-embeddable if and only if $k=2$.
	\end{theorem}
	
	\begin{proof}
		If $k=2$, then after parallel edges are suppressed one has
		\[
		Q_{2,m,b}\cong P_2\square C_m.
		\]
		This is an $\ell_1$-graph by Lemma~\ref{lem:product}.
		
		Assume now that $k\ge4$ and that $Q_{k,m,b}$ has a scale-$\lambda$ embedding.  The facial 4-cycles in the cylinder are isometric and give
		\[
		\ell(x_1x_2)=\ell(u_{1,1}u_{1,2})=\cdots=\ell(y_1y_2),
		\]
		with the evident omission of the interior terms when $m=2$.  The boundary 4-cycle $x_1x_2y_ky_1x_1$ is also isometric, so
		\[
		\ell(x_1x_2)=\ell(y_1y_k).
		\]
		Consequently $\ell(y_1y_2)=\ell(y_1y_k)$.  These are distinct incident edges because $k\ge4$, contradicting Lemma~\ref{lem:incident}.
	\end{proof}
	
	\subsection[The family Qkmh]{The family $Q_{k,m,h}$}
	
	We recall the construction of $Q_{k,m,h}$, with $k\ge1$ and $m\ge2$.  Start with a square grid $P_m\square P_m$.  Denote its four boundary paths by
	\[
	x_1x_2\cdots x_m,
	\quad y_1y_2\cdots y_m,
	\quad u_1u_2\cdots u_m,
	\quad v_1v_2\cdots v_m,
	\]
	where $x_m=y_1$, $x_1=u_m$, $v_1=u_1$, and $y_m=v_m$.  Add $m$ internally disjoint paths
	\[
	x_jz_{1,j}z_{2,j}\cdots z_{k,j}y_j\qquad (1\le j\le m),
	\]
	and also the horizontal paths $z_{i,1}z_{i,2}\cdots z_{i,m}$ for $1\le i\le k$.  Finally add the boundary edges
	\[
	u_1z_{1,m},\quad u_2z_{2,m},\quad\ldots,\quad z_{k-1,1}v_2,
	\quad z_{k,1}v_1,
	\]
	with the natural interpretation when $k=1$.  This gives the quadrilateral map $Q_{k,m,h}$ on the Klein bottle.  Let
	\[
	w_1,w_2,\ldots,w_m
	\]
	denote the path from $u_{m-1}$ to $y_2$, so $w_1=u_{m-1}$ and $w_m=y_2$.
	
	Figure~\ref{fig:qkmh} gives a schematic of the $h$-family and the exceptional graph $Q_{1,2,h}$.  The highlighted path in the left panel is the path $w_1,w_2,\ldots,w_m$ used in the label-propagation proof for $m\ge3$.
	
	\begin{figure}[H]
		\centering
		\begin{minipage}{0.62\textwidth}
			\centering
			\includegraphics[width=\linewidth]{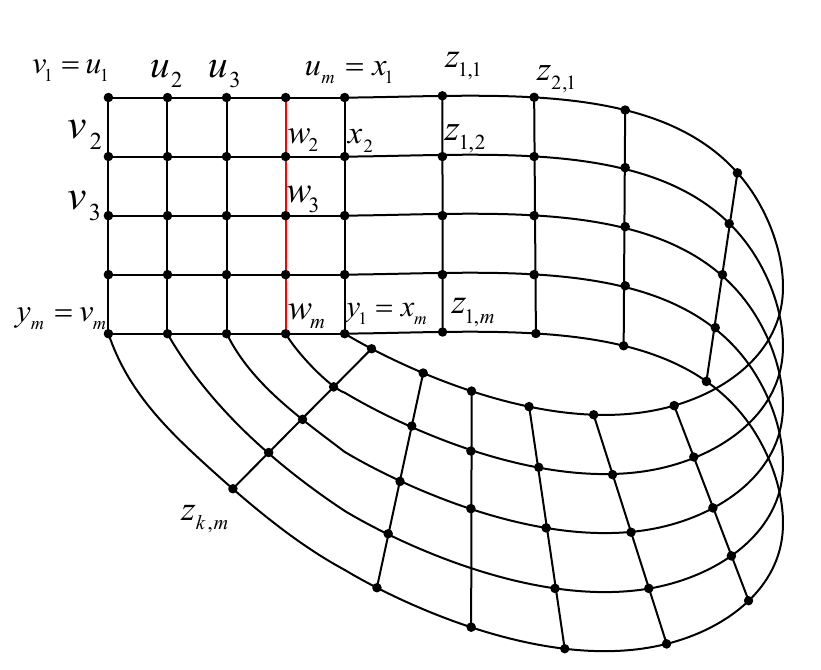}\\[-1mm]
			{\small (a) Schematic for $Q_{k,m,h}$.}
		\end{minipage}\hfill
		\begin{minipage}{0.32\textwidth}
			\centering
			\includegraphics[width=\linewidth]{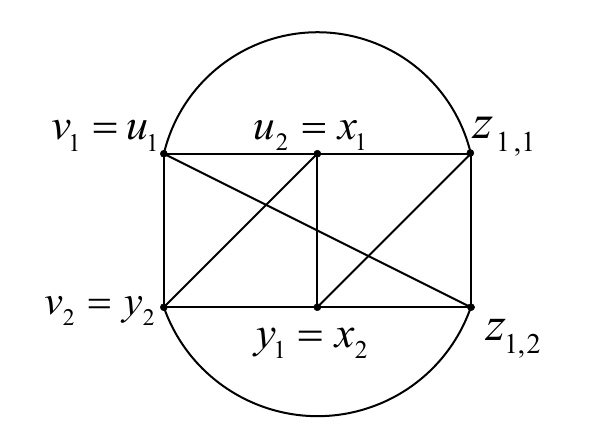}\\[-1mm]
			{\small (b) The exceptional graph $Q_{1,2,h}$.}
		\end{minipage}
		\caption{The Klein-bottle family $Q_{k,m,h}$.  Panel (a) shows the square grid together with the added $z$-paths and the distinguished path $w_1,\ldots,w_m$.  Panel (b) shows the small positive case, which is isomorphic to $Q_{3,2,1}$.}
		\label{fig:qkmh}
	\end{figure}
	
	\begin{theorem}\label{thm:qkmh}
		For $k\ge1$ and $m\ge2$, the graph $Q_{k,m,h}$ is $\ell_1$-embeddable if and only if $(k,m)=(1,2)$.
	\end{theorem}
	
	\begin{proof}
		The graph $Q_{1,2,h}$ is isomorphic to $Q_{3,2,1}$.  An explicit isomorphism is
		\[
		u_1=v_1\mapsto a_{0,0},\quad
		u_2=x_1\mapsto a_{0,1},\quad
		z_{1,1}\mapsto a_{0,2},
		\]
		\[
		y_2=v_2\mapsto a_{1,0},\quad
		y_1=x_2\mapsto a_{1,1},\quad
		z_{1,2}\mapsto a_{1,2}.
		\]
		Since $Q_{3,2,1}$ is $\ell_1$-embeddable by Proposition~\ref{prop:qkmr-positive}, so is $Q_{1,2,h}$.
		
		Now assume first that $k\ge2$ and $m=2$.  Define
		\[
		f(u_1=v_1)=a_{0,0},\quad f(u_2=x_1)=a_{0,1},\quad
		f(y_1=x_2)=a_{1,1},\quad f(y_2=v_2)=a_{1,0},
		\]
		and, for $1\le i\le k$,
		\[
		f(z_{i,1})=a_{0,i+1},\qquad f(z_{i,2})=a_{1,i+1},
		\]
		where the second coordinate is taken modulo $k+2$.  This is an isomorphism
		\[
		Q_{k,2,h}\cong Q_{k+2,2,2}.
		\]
		By Theorem~\ref{thm:qkmr}, the graph $Q_{k+2,2,2}$ is not $\ell_1$-embeddable; hence neither is $Q_{k,2,h}$.
		
		It remains to consider $m\ge3$.  Suppose, for a contradiction, that $Q_{k,m,h}$ has a scale-$\lambda$ embedding.  Facial 4-cycles are isometric, and repeated applications of Lemma~\ref{lem:face-labels} give
		\[
		\ell(w_1w_2)=\ell(x_1x_2)=\ell(z_{1,1}z_{1,2})=\cdots=\ell(z_{k,1}z_{k,2})=\ell(w_mx_m)=\cdots=\ell(w_2x_2).
		\]
		The first and last edges in this chain, $w_1w_2$ and $w_2x_2$, are distinct incident edges.  This contradicts Lemma~\ref{lem:incident}.  Therefore no $Q_{k,m,h}$ with $m\ge3$ is $\ell_1$-embeddable.
	\end{proof}
	
	The three preceding theorems give the Klein-bottle part of the classification.
	
	\begin{corollary}[Klein-bottle maps]\label{cor:klein}
		Among the essential regular quadrilateral maps on the Klein bottle, the $\ell_1$-embeddable graphs are precisely
		\[
		Q_{2,2,a},\qquad Q_{4,2,a},\qquad Q_{2,m,b},\qquad Q_{1,2,h}.
		\]
		Together with Lemma~\ref{lem:klein-reductions}, this completes the classification for all regular quadrilateral Klein-bottle maps.
	\end{corollary}
	
	\section{Conclusion}\label{sec:conclusion}
	
	We have determined the $\ell_1$-embeddability of the regular quadrilateral maps on the torus and the Klein bottle in the five essential families $Q_{k,m,r}$, $Q_{k,m,e}$, $Q_{k,m,a}$, $Q_{k,m,b}$, and $Q_{k,m,h}$.  The complete list of positive cases is
	\[
	Q_{k,m,0},\quad Q_{k,2,1},\quad Q_{k,1,e},\quad Q_{5,2,e},\quad Q_{2i,2,e}\ (i\ge2),
	\]
	for the torus and
	\[
	Q_{2,2,a},\quad Q_{4,2,a},\quad Q_{2,m,b},\quad Q_{1,2,h}
	\]
	for the Klein bottle.  All other members of these families fail to be $\ell_1$-embeddable.
	
	The classification shows that local square-grid geometry does not determine $\ell_1$-embeddability.  Untwisted product identifications preserve the two natural cut systems of the square grid, while most shifted or orientation-reversing identifications force incompatible Shpectorov labels after one travels around a non-contractible direction.  The exceptional positive cases survive either because the quotient is a Cartesian product, because it has width two and is isomorphic to a known $\ell_1$-graph, or because it collapses to a small complete graph.
	
	The toroidal family $Q_{k,m,e}$ also yields a concrete circulant consequence: within the two-generator subfamily $\Circ(k,\{\pm1,\pm m\})$ with $1\le m\le\lfloor k/2\rfloor$, the $\ell_1$-embeddable graphs are exactly the cycle cases $m=1$, the exceptional $K_5$, and the even squared-cycle family $\Circ(2i,\{\pm1,\pm2\})$.  It would be natural to ask whether similar label-propagation and hypermetric certificates can classify wider circulant or grid-quotient families.
	
	\section{Acknowledgements} This work was supported by the National Natural Science Foundation of China (12671419), the Natural Science Foundation of Shandong Province (ZR2024MA073), and the Institute for Basic Science (\mbox{IBS-R029-C4}). The authors would like to thank Professor Hong Liu for his valuable guidance during the preparation of this manuscript.
	
	\section{Conflict of Interest Statement} The authors declare that they have no conflict of interest.

	\section{Authorship} Guangfu Wang contributed to the conception and design of the study, supervised the research, provided methodological guidance, and revised the manuscript critically for important intellectual content. Wenchao Cong contributed to the analysis of the graph families, carried out the main proofs, and drafted the manuscript. Chunxiao Xu contributed to the analysis and interpretation of the results, checked the proofs, and helped revise the manuscript. All authors approved the final version of the manuscript, agreed to be accountable for all aspects of the work, and agreed on the order of authorship.

	\section{Data Availability Statement}
	Data sharing not applicable to this article as no datasets were generated or analyzed during the current study.
	
	\section{Declaration of generative AI use}	
	
	The authors declare that they did not use any generative artificial intelligence (AI) or AI-assisted technologies in the preparation   of this manuscript.

\end{document}